\documentclass[12pt,dvipsnames]{amsart}

\usepackage{amsmath,mathdots,verbatim,mathtools,nccmath}
\usepackage[T1]{fontenc}
\usepackage[colorlinks=true]{hyperref}
\usepackage[utf8]{inputenc}
\usepackage[sc]{mathpazo}
\usepackage{pstricks}
\usepackage[margin = 2.5cm]{geometry}
\usepackage{tikz}
\usepackage{ytableau}
\usepackage{color}
\usepackage{MnSymbol}
\usepackage{graphics}
\usepackage{graphicx}
\usepackage{nicematrix}
\usepackage{scalerel}
\usepackage{enumitem}
\usepackage[normalem]{ulem}
\newtheorem{theorem}{Theorem}

\newtheorem{lemma}[theorem]{Lemma}

\newtheorem{conj}[theorem]{Conjecture}

\theoremstyle{definition}

\newtheorem{remark}[theorem]{Remark}

\DeclareMathOperator{\sgn}{sgn}

\DeclareMathOperator{\diag}{diag}

\usepackage{enumitem}

\newcommand{\asym}{\mathbf{ASym}}

\newcommand{\sym}{\mathbf{Sym}}

\usepackage{bbm}

\newmuskip\pFqskip
\mathchardef\pFcomma=\mathcode`, % keep a copy of the comma

\title[]{A new antisymmetrizer-to-determinant formula and a conjecture of Colomo and Pronko}

\author{Ilse Fischer and Markus Reibnegger}
\thanks{The authors acknowledge the financial support from the Austrian Science Fund (FWF) grant 10.55776/F1002.}

\begin{document}

\begin{abstract}
Identities asserting that an antisymmetrizer admits a determinantal expression are central to several proofs of counting formulas for alternating sign matrices. Antisymmetrizers appear also frequently in symmetric function theory as, for instance, the Hall–Littlewood polynomials can be realized via antisymmetrizers. In this paper, we establish a new identity of this type, thereby proving a conjecture of Lukas Riegler and one of the authors. We also elaborate on a related antisymmetrizer that may be pivotal in resolving a beautiful conjecture of Colomo and Pronko, and we formulate a version of their conjecture tailored to the suggested approach.
\end{abstract}

\maketitle

\section{Introduction} 

The \emph{antisymmetrizer} of a, say, rational function $F(X_1,\ldots,X_n)$ is defined as 
$$
\asym_{X_1,\ldots,X_n} \left[ F(X_1,\ldots,X_n) \right] = \sum_{\sigma \in \mathfrak{S}_n} \sgn \sigma \cdot F(X_{\sigma(1)},\ldots,X_{\sigma(n)}),
$$
where $ \mathfrak{S}_n$ denotes the symmetric group of order $n$.
\emph{Schur polynomials} are (up to sign) precisely the antisymmetrizers of monomials divided by the Vandermonde determinant 
$\prod_{1 \le i < j \le n}(X_i-X_j)$. However, there are interesting antisymmetrizers that are not (yet?) expressible by a single determinant. For instance, in \cite[(2.8)]{gog} it was shown that 
\begin{equation}
\label{asymmt}
\frac{\asym_{X_1,\ldots,X_n} \left[  
\prod\limits_{i=1}^n X_i^{b_i}  
\prod\limits_{1 \le i < j \le n} (1-X_i + X_i X_j) \right]}{\prod\limits_{1 \le i < j \le n} (X_j-X_i)}
\end{equation} 
evaluated at $(X_1,\ldots,X_n)=(1,\ldots,1)$ is the number of \emph{monotone triangles} with bottom row $b_1,\ldots,b_n$.
In fact, if we do not specialize the $X_i$, then the 
expression is a multivariate generating function of monotone triangles as was shown in \cite{FS23}. By introducing 
three more parameters, this generating function can easily be lifted to a generating function of monotone triangles that generalize Schur polynomials. Antisymmetrizers also appear frequently in 
the theory of symmetric functions. For instance note that the Hall-Littlewood polynomial
$P_{\lambda}(X_1,\ldots,X_n;t)$ of the partition $\lambda$ can also be written as 
$$
P_{\lambda}(X_1,\ldots,X_n;t) = \frac{1}{v_{\lambda}(t)} 
\frac{\asym_{X_1,\ldots,X_n} \left[ \prod_{i=1}^n X_i^{\lambda_i} \prod_{1 \le i < j \le n} (X_i- t X_j)\right]}{\prod_{1 \le i < j \le n} (X_i-X_j)}
$$ 
where 
$$
v_{\lambda}(t) = \frac{(t;t)_{n-\ell(\lambda)}}{(1-t)^{n-\ell(\lambda)}}
\prod_{i \ge 1} \frac{(t;t)_{m_i(\lambda)}}{(1-t)^{m_i(\lambda)}},
$$
$\ell(\lambda)$ denotes the number of (non-zero) parts of $\lambda$, $m_i(\lambda)$ denotes 
the number of parts equal to $i$ and $(a;q)_n= \prod_{k=0}^{n-1} (1- a q^k)$ is the $q$-Pochhammer symbol. There exist common generalizations of the symmetric functions in \eqref{asymmt} and the Hall-Littlewood polynomials that are also expressed in terms of 
antisymmetrizers, see \cite{FG26}; in fact the above mentioned three parameter extension of \eqref{asymmt} is already a common generalization.

\bigskip

Monotone triangles are \emph{Gelfand-Tsetlin patterns} with strictly increasing rows and Gelfand-Tsetlin patterns are 
triangular arrays of integers of the following form 
\begin{equation*}
\begin{array}{ccccccccccc}
& & & & & & a_{1,1} & & & &  \\
& & & & & a_{2,1} & & a_{2,2} & & &  \\
& & & & a_{3,1} & & a_{3,2} & & a_{3,3} & &  \\
& & & \dots & & \dots & & \dots & & \dots &  \\
& & a_{n,1} & & a_{n,2} & & \dots & & \dots & & a_{n,n} 
\end{array}
\end{equation*}
that are weakly increasing in $\nearrow$- and $\searrow$-direction.
Monotone triangles with bottom row 
$1,2,\ldots,n$ are in easy bijective correspondence with $n \times n$ \emph{alternating sign matrices} (ASMs). Alternating sign matrices 
are square matrices with enries in $\{0,\pm 1\}$ such that, in each row and each column, non-zero entries alternate and add up to $1$, see \cite{Bre00} for an account of the early history of ASMs.
Below is a monotone triangle and its corresponding ASM: We have $j$ in row $i$ of the monotone triangle if and only if the top $i$ entries in column $j$ of the ASM sum to $1$.
$$
\begin{array}{ccccccccccc}
& & & & & & 3 & & & &  \\
& & & & & 2 & & 4 & & &  \\
& & & & 1 & & 4 & & 5 & &  \\
& & & 1 & & 2 & & 4 & & 5 &  \\
& & 1 & & 2 & & 3 & & 4 & & 5 
\end{array}
\quad 
\Leftrightarrow 
\quad 
\begin{pmatrix}
0 & 0 & 1 & 0 & 0 \\
0 & 1 & -1 & 1 & 0 \\
1 & -1 & 0 & 0 & 1 \\
0 & 1 & 0 & 0 & 1 \\
0 & 0 & 1 & 0 & 1
\end{pmatrix} 
$$

In the special case $(b_1,\ldots,b_n)=(1,2,\ldots,n)$, the  
antisymmetrizer-to-determinant lemma below (Lemma~\ref{first}) can be applied to express the antisymmetrizer in \eqref{asymmt} as a determinant.  Applying a formula to transform bialternants to Jacobi-Trudi-type determinants (see \eqref{limit1}), the determinant from Lemma~\ref{first} can be used to reduce the enumeration of $n \times n$ ASMs to the evaluation of the \emph{descending plane partition} determinant. The latter was accomplished by Andrews in \cite{And79}. More details about this are provided in Remark~\ref{andrews}. Note that this strategy provides an alternative proof of the ASM enumeration formula, which is a notoriously difficult task---see \cite{Zei96} for the first proof.

\begin{lemma} 
\label{first}
Let $n$ be a positive integer and $f(X)$ be a, say, rational function in $X$, then 
$$
\asym_{X_1,\ldots,X_n} \left[ \prod_{1 \le i \le j \le n} (f(X_j)-X_i) \right] = \det_{1 \le i, j \le n} \left(f(X_i)^j-X_i^j\right).
$$
\end{lemma}

A concise proof of the lemma can be found in \cite[Lemma~3.2]{FisHo25}, while a combinatorial proof is given in \cite[Appendix~A]{FS24}. A generalization (including a proof) is presented in Section~\ref{cauchysec}.

The main purpose of the current paper is to provide a new, related antisymmetrizer-to-determinant formula, thereby proving Conjecture~1.1 from \cite{FisRie}. This conjecture was also presented at the 2014 Oberwolfach Meeting on ``Enumerative Combinatorics'', see \cite{OW14}.  The conjecture is in terms of the symmetrizer,  
which is defined as 
$$
\sym_{X_1,\ldots,X_n} \left[ F(X_1,\ldots,X_n) \right] \coloneq \sum_{\sigma \in \mathfrak{S}_n} F(X_{\sigma(1)},\ldots,X_{\sigma(n)}).
$$
However, note also that 
$$
\sym_{X_1,\ldots,X_n} \left[ \frac{G(X_1,\ldots,X_n)}{\prod_{1 \le i < j \le n} (X_i-X_j)} \right] = 
\frac{\asym_{X_1,\ldots,X_n} \left[ G(X_1,\ldots,X_n) \right]}{\prod_{1 \le i < j \le n} (X_i-X_j)}
$$
since 
$$
 \prod_{1 \le i < j \le n} (X_{\sigma(i)}-X_{\sigma(j)})=\sgn \sigma \cdot \prod_{1 \le i < j \le n} (X_i-X_j),
$$
and thus the conjecture can also be formulated in terms of the antisymmetrizer.

\begin{conj}{\cite[Conjecture 1.1]{FisRie}}
\label{mainconj} 
For non-negative integers $s,t$, define the following rational function in $X_1,\ldots,X_{s+t-1}$
\begin{multline*} 
P_{s,t}(X_1,\ldots,X_{s+t-1})=\prod_{i=1}^{s} X_i^{2s-2i-t+1} (1-X_i^{-1})^{i-1} \prod_{i=s+1}^{s+t-1} X_i^{2i-2s-t} (1-X_i^{-1})^s \\ \times
\prod_{1 \le p < q \le s+t-1} \frac{1-X_p+X_p X_q}{X_q-X_p}
\end{multline*} 
and let $R_{s,t}(X_1,\ldots,X_{s+t-1}) \coloneq \sym_{X_1,\ldots,X_{s+t-1}} \left[ P_{s,t}(X_1,\ldots,X_{s+t-1}) \right]$. If $s \le t$, then 
$$
R_{s,t}(X_1,\ldots,X_{s+t-1}) = R_{s,t}(X_1,\ldots,X_{i-1},X_i^{-1},X_{i+1},\ldots,X_{s+t-1})
$$ 
for all $i \in \{1,2,\ldots,s+t-1\}$. 
\end{conj}

The significance of the conjecture stems from the fact that it implies the following refined enumeration of vertically symmetric ASMs. Vertically 
symmetric ASMs exist only for odd orders as the unique $1$ in the top row needs to be situated in the center, which therefore needs to exist. For orders greater than $1$, they have two symmetrically distributed $1$'s in the second row. In \cite{half}, it was conjectured 
that the number of $(2n+1) \times (2n+1)$ vertically symmetric ASMs that have a $1$ in row $2$ and column $i$ where $1 \le i \le n$ is 
\begin{equation}
\label{refined}  
\frac{\binom{2n+i-2}{2n-2} \binom{4n-i-1}{2n-1}}{\binom{4n-2}{2n-1}} 
\prod_{j=1}^{n-1} \frac{(3j-1)(2j-1)!(6j-3)!}{(4j-2)!(4j-1)!}, 
\end{equation} 
and, in \cite{FisRie}, it was shown that Conjecture~\ref{mainconj} implies this refined enumeration. Note that meanwhile the refined enumeration of vertical symmetric ASMs has been proven by different means in \cite{FisSai}, while Conjecture~\ref{mainconj} remained open.

Clearly, the conjecture also implies 
$$
R_{s,t}(X_1,\ldots,X_{s+t-1}) = R_{s,t}(X_1^{-1},\ldots,X_{s+t-1}^{-1})
$$
and it is only this consequence that is needed to establish the formula in \eqref{refined} for the refined enumeration of vertically symmetric ASMs as described above. Also in \cite{FisRie} it was shown that it suffices to prove Conjecture~\ref{mainconj} for $s=t$ and $s+1=t$. In this paper, we will not rely on this result, however, we will still first establish the 
case $s= \lfloor \frac{n+1}{2} \rfloor$ of Theorem~\ref{antisymtodet} (see Lemma~\ref{basecase}) as the base case of an induction, and this does correspond to the cases $s=t$ and $s+1=t$ of Conjecture~\ref{mainconj}.

\medskip

In order to prove Conjecture~\ref{mainconj}, we compute for non-negative integers $n,s$ with $0 \le s \le \lfloor \frac{n+1}{2} \rfloor$ 
the following antisymmetrizer 
\begin{equation}
\label{D}
D_s(X_1,\ldots,X_n)=\asym_{X_1,\ldots,X_n} \left[ \prod_{i=1}^s p(X_i)^{s-i} q(X_i)^{2i-2s-1} 
\prod_{1 \le i \le j \le n} (p(X_j)-q(X_i)) \right], 
\end{equation}
where $p(X)=X(X-1)$ and $q(X)=p(X^{-1})=-p(X)X^{-3}$.
In the following theorem, we are using Iverson notation: We set $[\text{statement}]=1$ if the statement is true and 
$[\text{statement}]=0$ otherwise. 

\begin{theorem}
\label{antisymtodet} 
For non-negative integers $n,s$ with $0 \le s \le \lfloor \frac{n+1}{2} \rfloor$, we have 
$$
D_s(X_1,\ldots,X_n) = 
(-1)^s \det_{1 \le i,j \le n} \left( p(X_i)^s \left( p(X_i)^{j-(1+[j<2s])s}-q(X_i)^{j-(1+[j<2s])s} \right) \right).
$$
\end{theorem} 

Although the proof of this theorem spans several pages (see Section~\ref{proof}), the real difficulty lies in identifying rational functions whose antisymmetrizers admit a determinantal expression, together with the appropriate Ansatz for that determinant. This underscores that our understanding is still incomplete. In Section~\ref{colpro}, we discuss a related antisymmetrizer for which we have not yet found a determinantal representation, despite evidence suggesting that one should exist.

It is worth noting that experiments suggest that the result is also true if we set $p(X)=X(X-\rho)$ and $q(X)=-p(X)X^{-3}$ where 
$\rho$ is a third root of unity, and it seems conceivable that the proof can be generalized to this setting.

\bigskip

{\bf Outline of the paper.} In Section~\ref{colpro}, we elaborate on a related conjecture of Colomo and Pronko alongside a possible approach to prove it. In particular, we offer a formulation of the conjecture that relates the matrices appearing in the conjecture to a certain antisymmetrizer that is central to this approach, and thus could be helpful in proving it. In Section~\ref{proof}, we prove the main result Theorem~\ref{antisymtodet}. There we also argue why Conjecture~\ref{mainconj} is a consequence of Theorem~\ref{antisymtodet}.
In Section~\ref{cauchysec}, we present a common generalization of Lemma~\ref{first} and the Cauchy determinant evaluation and pose it as an open problem to find a similar generalization of Theorem~\ref{antisymtodet}.

\section{The conjecture of Colomo and Pronko and a related antisymmetrizer} 
\label{colpro}

In this section, we discuss a fascinating conjecture of Colomo and Pronko 
\cite{CP24,CP25} and propose a possible approach to prove it that involves an antisymmetrizer that has some similarities with the one that is dealt with in Theorem~\ref{antisymtodet}.

For integers $n,s$ with $0 \le s \le \lfloor \frac{n}{2} \rfloor$, consider 
$$F_{n,s}= \# \text{ of $n \times n$ ASMs that have only zeros in the bottom-left $s \times s$ corner}. 
$$
Colomo and Pronko provide a conjectural 
$s \times s$ determinant for the number of such ASMs. Enumerations of classes of ASMs 
(such as several symmetry classes, see \cite{Kup00}) are so far typically given by product formulas. The reason is 
that such product formulas are---in contrast to the formula conjectured by Colomo and Pronko---much easier 
to detect because they can be guessed from small data by some type of interpolation: In these cases, the quotient sequence of the numbers is often expressible by a simple rational function. Of course proving such formulas is then also a different task and typically quite involved, however, a fascinating feature of the formula of Colomo and Pronko is that the guessing part is a major challenge too. 

Concretely, 
let $A_n$ denote the number of $n \times n$ ASMs and $A_{n,k}$ the number of $n \times n$ ASMs, where the unique $1$ in the top 
row is located in the $k$-th column, and set 
\begin{multline*} 
m_{i,j} =  \sum_{0 \le l \le k \atop {1 \le p \le n-s+i \atop 1 \le q \le n-s+j}} 
(-1)^{i+j+k+p+q} \binom{k}{l} \left( \binom{i-1}{i+l-k-p} - (-1)^i \binom{i-1}{i+l-k-p-1} \right) \\
\times 
\left( \binom{j-1}{j-l-q} + (-1)^j \binom{j-1}{j-l-q-1} \right) \frac{A_{n-s+i,p} A_{n-s+j,q}}{A_{n-s+i} A_{n-s+j-1}}  
\end{multline*}
for $i,j=1,\ldots,s$ as well as $M=(m_{i,j})_{1 \le i,j \le s}$, then one version of their conjecture (see \cite[Conjecture~1]{CP25}) is that 
$$
F_{n,s} = A_n \det (1-M).
$$

\bigskip

Using the well-known bijection between monotone triangles with bottom row $1,2,\ldots,n$ and 
$n \times n$ ASMs, $F_{n,s}$ is also the number of monotone triangles with bottom row $1,2,\ldots,n$ such that 
the bottom $s$ elements of the $i$-th $\nearrow$-diagonal of the monotone triangle are $i$ for $1 \le i \le s$, 
counting $\nearrow$-diagonals from left to right. 

Below is an example of such an ASM for $s=3$ and $n=8$ on the 
left and the corresponding monotone triangle on the right, with the fixed entries highlighted in orange.
$$
\begin{pmatrix}
0 & 0 & 0 & 1 & 0 & 0 & 0 & 0 \\
0 & 0 & 1 & -1 & 0 & 1 & 0 &0 \\
0 & 1 & -1 & 0&  1 &-1& 1& 0 \\
1 & -1 & 1 & 0 & -1& 1& -1 & 1 \\
0 & 1 & 0 &  0& 0 &-1&  1 & 0\\
\color{orange}0 & \color{orange}0 & \color{orange}0 & 0 & 0 & 1 &0 & 0 \\
\color{orange}0 & \color{orange}0 & \color{orange}0 & 1 & 0 &0 & 0 & 0 \\
\color{orange}0 & \color{orange}0 & \color{orange}0 & 0 & 1 & 0 & 0 & 0 
\end{pmatrix} 
\Leftrightarrow 
\begin{array}{ccccccccccccccccccccccc}
&&& & & & & & & & 4 & & & & & &&&  \\
&&& & & & & & & 3 & & 6 & & & & &&&  \\
&&& & & & & & 2 & & 5 & & 7 & & & &&&  \\
&&& & & & & 1 & & 3 & & 6 & & 8 & & &&&  \\
&&& & & & 1 & & 2 & & 3 & & 7 & & 8 & &&& \\
&&&&&\color{orange}1&&\color{orange}2&&\color{orange}3&&6&&7&&8&& \\
&&&&\color{orange}1&&\color{orange}2&&\color{orange}3&&4&&6&&7&&8\\
&&&\color{orange}1&&\color{orange}2&&\color{orange}3&&4&&5&&6&&7&&8
\end{array}
$$
Corollary~3 from  \cite{gog} alongside with (2.8) then implies the following. Let 
\begin{equation}
\label{fasym} 
F_s(X_1,\ldots,X_n)= \asym_{X_1,\ldots,X_n} \left[ \prod_{i=1}^s (1-X_i)^s 
\prod_{1 \le i \le j \le n} (X_i^{-1}+ X_j-1) \right],
\end{equation}then 
$$
F_{n,s} = \left. \frac{F_s(X_1,\ldots,X_n)}{\prod_{1 \le i < j \le n} (X_j-X_i)}  \right|_{(X_1,\ldots,X_n)=\mathbf{1}}.
$$
Comparing this to \eqref{D}, we see that the rational functions to which we apply the antisymmetrizer enjoy some 
structural similarities: On the one hand, we have the factor $\prod_{1 \le i \le j \le n} (p(X_j)-q(X_i))$ in \eqref{D} and similarly 
the factor $\prod_{1 \le i \le j \le n}(X_i^{-1}+X_j-1)$ in \eqref{fasym}, which is also a common feature with Lemma~\ref{first} where we have 
$\prod_{1 \le i \le j \le n} (f(X_j)-X_i)$ (and in this case the product comprises the whole expression), however, in contrast to Lemma~\ref{first}, the 
first $s$ variables $X_1,\ldots,X_s$ play a different role than the variables $X_{s+1},\ldots,X_n$ in \eqref{D} as well as 
in \eqref{fasym}. Also, in both cases, the relevant range for $s$ is between $0$ and (approximately) $\frac{n}{2}$. 

To prove the conjecture of Colomo and Pronko, it would suffice to provide a formula for 
$F_s(X_1,\ldots,X_n)$ similar to the one for $D_s(X_1,\ldots,X_n)$ in Theorem~\ref{antisymtodet}. 
While we have not succeeded in this task so far, we offer a formulation of their conjecture that
is related to $F_s(X_1,\ldots,X_n)$, which may help to prove the conjecture using this approach. 

Note that there is the following simple approach to compute a generalization of $F_{s}(X_1,\ldots,X_n)$ recursively. Namely, let 
$$
G_{s,t}(X_1,\ldots,X_n)= \asym_{X_1,\ldots,X_n} \left[ \prod_{i=1}^t (1-X_i)^s 
\prod_{1 \le i \le j \le n} (X_i^{-1}+ X_j-1) \right],
$$
so that $F_s(X_1,\ldots,X_n)=G_{s,s}(X_1,\ldots,X_n)$, 
then we have  
$$
G_{s,t}(X_1,\ldots,X_n) = \sum_{i=1}^n (-1)^{i+1} (1-X_i)^{s} \prod_{j=1}^n (X_i^{-1}+X_j-1) 
G_{s,t-1}(X_1,\ldots,\widehat{X_i},\ldots,X_n)
$$
for $t>0$, and 
$$
G_{s,0}(X_1,\ldots,X_n) = \sum_{i=1}^n (-1)^{i+1} \prod_{j=1}^n (X_i^{-1}+X_j-1) 
G_{s,0}(X_1,\ldots,\widehat{X_i},\ldots,X_n).
$$
See also the proof of  Lemma~\ref{basecase}, where such a recursion is justified in more detail.

We consider the case $s=0$ and note that 
\begin{multline*} 
F_{n,0} = \left. \frac{\asym_{X_1,\ldots,X_n} \left[ 
\prod_{1 \le i \le j \le n} (X_i^{-1}+ X_j-1) \right]}{\prod_{1 \le i < j \le n} (X_j-X_i)}  \right|_{(X_1,\ldots,X_n)=\mathbf{1}} \\
= 
\left. \frac{\asym_{X_1,\ldots,X_n} \left[ 
\prod_{1 \le i \le j \le n} ((1+X_i)^{-1}+ X_j) \right]}{\prod_{1 \le i < j \le n} (X_j-X_i)}  \right|_{(X_1,\ldots,X_n)=\mathbf{0}}.
\end{multline*}
By Lemma~\ref{first}, this is equal to 
\begin{equation}
\label{specialcase} 
\left. \frac{\det_{1 \le i,j \le n} \left(X_i^j - (-1-X_i)^{-j}\right)}{\prod_{1 \le i < j \le n} (X_j-X_i)}  
\right|_{(X_1,\ldots,X_n)=\mathbf{0}}.
\end{equation} 
Next we use the following formula to transform a bilaternant into a Jacobi-Trudi type formula: Suppose $f_j(X)$ are formal Laurent series and $f_{j,k}$ is the coefficient of $X^k$ in $f_j(X)$, then 
\begin{equation}
\label{limit1g} 
\frac{\det_{1 \le i,j \le n} (f_j(X_i))}{\prod_{1 \le i < j \le n}(X_j-X_i)} = \det_{1 \le i,j \le n} \left( \sum_{k \in \mathbb{Z}} f_{j,i+k-1} h_k(X_1,\ldots,X_n) \right),
\end{equation} 
where $h_k(X_1,\ldots,X_n)$ denotes the $k$-th complete homogeneous symmetric function, which are extended as follows to negative integers.
$$
h_{k}(X_1,\ldots,X_n) = 
\begin{cases}
0, &  -n < k < 0, \\
(-1)^{n+1} X_1^{-1}  \dots X_n^{-1} h_{-k-n}(X_1^{-1},\ldots,X_n^{-1}), & k \le -n.
\end{cases} 
$$
For a proof see, for instance, 
\cite[Lemma~7.2]{FS23}.
In particular, if $f_j(X)$ are formal power series, then 
\begin{equation}
\label{limit1}
\left. \frac{\det_{1 \le i,j \le n} (f_j(X_i))}{\prod_{1 \le i < j \le n}(X_j-X_i)}  \right|_{(X_1,\ldots,X_n)=\mathbf{0}}= \det_{1 \le i,j \le n} \left( f_{j,i-1} \right).
\end{equation} 
We apply this to $f_j(X)=X^j - (-1-X)^{-j}$, see \eqref{specialcase}, in which case we have 
$f_{j,i} = \delta_{i,j}+(-1)^{i}\binom{-i-1}{j-1}=\delta_{i,j}+(-1)^{i+j-1}\binom{i+j-1}{j-1}$ by the Binomial Theorem, so that 
$
\det_{1 \le i,j \le n} \left( f_{j,i-1} \right) = \det_{1 \le i,j \le n} \left( S_n + B_n\right), 
$
where 
$$
S_n = \left(\delta_{i-1,j}  \right)_{1 \le i,j \le n} \quad \text{and} \quad B_n= \left((-1)^{i+j}\binom{i+j-2}{j-1} \right)_{1 \le i,j \le n}.
$$
We also set $M_n=(S_n + B_n)^T=S_n^{T} + B_n$.

\begin{remark}
\label{andrews}
\it
As $F_{n,0}$ is the number of $n \times n$ ASMs, it now follows that, in order to prove the enumeration 
formula for $n \times n$ ASMs, it suffices to show that 
$
\det (M_n) = \prod_{i=0}^{n-1} \frac{(3i+1)!}{(n+i)!}. 
$
Andrews showed in \cite{And79}  
\begin{equation}
\label{andrews} 
\det_{1 \le i,j \le n-1} \left( \delta_{i,j} + \binom{i+j}{j-1} \right) = \prod_{i=0}^{n-1} \frac{(3i+1)!}{(n+i)!}.
\end{equation} 
Note that the evaluation of $\det(M_n)$ can be reduced to the evaluation of 
the determinant in \eqref{andrews} as 
$$
\left( \binom{i-1}{j-1} \right)_{1 \le i,j \le n} \cdot 
\left( \delta_{i-1,j} + (-1)^{i+j} \binom{i+j-2}{j-1} \right)_{1 \le i, j \le n} \cdot 
\left( \binom{j-1}{i-1} \right)_{1 \le i,j \le n} =
\left( \delta_{i,j} + \binom{i+j-2}{j-2} \right)_{1 \le i, j \le n}, 
$$
which follows from the Chu-Vandermonde summation, 
$$
\det_{1 \le i, j \le n} \left( \binom{i-1}{j-1} \right) = \det_{1 \le i, j \le n} \left( \binom{j-1}{i-1} \right)=1
$$ 
(the latter identities follow as the matrices are triangular with $1$'s on the main diagonal) 
and 
$$
\det_{1 \le i,j \le n-1} \left( \delta_{i,j} + \binom{i+j}{j-1} \right) = 
\det_{1 \le i,j \le n} \left( \delta_{i,j} + \binom{i+j-2}{j-2} \right)
$$
(by expanding the second determinant along the first column). 

\end{remark} 

Our crucial observation is that the conjecture of Colomo and Pronko can be expressed in terms of determinants of matrices that appear in the LU-decomposition of $B_n$ and $M_n$. 
The LU-decomposition of $B_n$ is rather simple and follows from the Chu-Vandermonde summation. 
(This decomposition is also behind the computation in the remark.) Setting 
$$
LB_n= \left((-1)^{i+j} \binom{i-1}{j-1} \right)_{1 \le i,j \le n} \quad \text{and} \quad 
UB_n= LB_n^T, 
$$
we have 
$$
B_n = LB_n \cdot UB_n.
$$

We let $L_n,U_n$ be the matrices in the LU-decomposition of $M_n$. We define 
\begin{align*} 
A_n & = ((-1)^i \delta_{i,j})_{1 \le i,j \le n}, \\
H_n &= A_n \cdot U_n \cdot A_n \cdot (I_n - S_n^T)^{-1}, \\
G_n &= \diag_{1 \le i \le n} ((H_n)^{-1}_{i,i}) \cdot H_n.
\end{align*} 

\begin{remark} 
\it There seems to be the following formula for $H_n$
$$
H_n =\left[ \left( \sum_{k=0}^{j-1} (-1)^k \binom{k-1+j-i}{j-i} \frac{(j-k)_{2k}}{(1)_k (2j)_k} \right)_{1 \le i,j \le n} \right]^{-1},
$$
which could be extracted from the formulation of the conjecture that Colomo and Pronko gave. 
We also seem to have the following
$$
G_n = A_n \cdot L_n^{T}  \cdot A_n \cdot  (I_n-S_n^T)^{-1} \cdot (I_n - S_n^T +  \left[ S_n^T \right]^2).
$$
\end{remark} 

We use the following notation: Suppose $M$ is a matrix and $s$ is a non-negative integer, then $(M)_s$ denotes the $s \times s$ submatrix of $M$ in the 
bottom right corner. Here is the reformulated conjecture of Colomo and Pronko. (We do not prove the equivalence 
as we do not see any value in this.)

\begin{conj} Suppose $n$ is a positive integer and $s$ is a non-negative integer with 
$0 \le s \le \lfloor \frac{n}{2} \rfloor$. Then the number $F_{n,s}$ of $n \times n$ ASMs that have only zeros in the bottom left $s \times s$ corner is
$$
\det(H_{n-s}) \cdot \det \left[ (G_n^T)_s \cdot (H_n)_s - (I_n+(-1)^n A_n \cdot LB_n  \cdot S_n)_s \cdot (I_n+(-1)^{n+1} S_n^T \cdot UB_n \cdot A_n)_s \right].
$$
\end{conj}
Note that $\det(H_{n-s})$ is the number of $(n-s) \times (n-s)$ ASMs. This is because 
$A_{n} = \det(M_{n}) = \det(U_{n}) = \det(H_{n})$ by the definition of the $H_n$, the multiplicativity of the determinant and 
$\det(A_{n})=\det \left( (I_n-S_n^T)^{-1} \right)=1$.

\section{Proofs of Theorem~\ref{antisymtodet} and Conjecture~\ref{mainconj}}
\label{proof} 
 
First we argue why Theorem~\ref{antisymtodet} implies Conjecture~\ref{mainconj}. Setting 
$$
B_{s}(X_1,\ldots,X_n) = \prod_{i=1}^s p(X_i)^{s-i} q(X_i)^{2i-2s-1} \prod_{1 \le i \le j \le n} (p(X_j)-q(X_i)),
$$
which is just the rational function of which $D_s(X_1,\ldots,X_n)$ is the antisymmetrizer,
it can easily be seen that 
\begin{multline*} 
\frac{B_{s}(X_1,\ldots,X_n)}{P_{s,n+1-s}(X_1,\ldots,X_n)} = (-1)^{s+n+s n} 
\prod_{i=1}^n X_i^{2s-n+1} (X_i+X_i^{-1}-1) (1+X_i^{-1}) (1-X_i)^{-s+1} \\ \times 
\prod_{1 \le i < j \le n} (1-X_i X_j)(X_i-X_j).
\end{multline*} 
This quotient has two properties that are important for us: First, as the quotient is symmetric, it can be taken outside any symmetrizer or antisymmetrizer. Second, if we replace $X_i$ by $X_i^{-1}$ for any fixed $i \in \{1,2,\ldots,n\}$, then this leaves the polynomial invariant up to the factor $(-1)^{s+1} X_i^{3s}$. Now the implication follows as the determinantal expression in Theorem~\ref{antisymtodet} for $D_s(X_1,\ldots,X_n)$ is also invariant up to the factor $(-1)^{s+1} X_i^{3s}$ if we replace $X_i$ by $X_i^{-1}$ for any fixed $i \in \{1,2,\ldots,n\}$.

\bigskip

We prove Theorem~\ref{antisymtodet} by induction with respect to $n$. 
The base case $n=2s-1$ follows from the following lemma. 
Note that in the case $n$ is odd, it can easily be seen that the determinant in Lemma~\ref{basecase} agrees 
with the determinant in Theorem~\ref{antisymtodet} for $s=\lfloor \frac{n+1}{2} \rfloor$. This is not so clear for 
$n$ is even, however, this case is not needed for the base case of Theorem~\ref{antisymtodet}.  

\begin{lemma}
\label{basecase} 
For any positive integer $n$, we have   
$$D_{\lfloor \frac{n+1}{2} \rfloor}(X_1,\ldots,X_n) = 
\det_{1 \le i, j \le n} \left( p(X_i)^{\lfloor \frac{n+1}{2} \rfloor} \left( q(X_i)^{-j} - p(X_i)^{-j} \right) \right).$$
\end{lemma} 

Also Lemma~\ref{basecase} will be proved by induction with respect to $n$ and here we need to distinguish between the case $n$ is even and $n$ is odd. The case $n$ is odd will be rather simple, while the case $n$ is even requires the following 
lemma, which involves the elementary symmetric functions defined as 
$$
e_k(X_1,\ldots,X_n) = \sum_{1 \le i_1 < i_2 < \ldots < i_k \le n} X_{i_1} X_{i_2} \cdots X_{i_k}.
$$

\begin{lemma}
\label{coeffbase}
For any positive even integer $n$ and any $i \in [n]$, we have 
$$
\sum_{j=1}^n \sum_{k=0}^n a_{j,k} e_k(q(X_1),\ldots,q(X_n))p(X_i)^{\frac{n}{2}}(q(X_i)^{-j}-p(X_i)^{-j}) = 
\prod_{k=1}^n (p(X_i) - q(X_k)),
$$ 
where 
$$
a_{j,k} = 
\begin{cases} (-1)^k \left( \binom{\frac{n}{2}+j-k}{n-2k-j} + \binom{\frac{n}{2}+j-k-1}{n-2k-j-1}\right), & \text{if 
$k \le \frac{n}{2}$}, \\
                       (-1)^{k+1} \left[k=\frac{n}{2}+j\right], & \text{otherwise}.
\end{cases} 
$$    
\end{lemma}

\begin{proof} 
Expanding the product on the right-hand side in the statement of the lemma yields
$$
\sum_{k=0}^n (-1)^k p(X_i)^{n-k} e_k(q(X_1),\ldots,q(X_n)).
$$
To simplify notation, we write $e_k$ for $e_k(q(X_1),\ldots,q(X_n))$ and $e^{(i)}_k$ for the same quantity but where $q(X_i)$ is 
omitted. Using the identity $e_k=e_k^{(i)}+q(X_i) e_{k-1}^{(i)}$ 
and the fact that the $e^{(i)}_k$'s are linearly independent for fixed $i$, reduces the problem to showing that 
the coefficients of the $e^{(i)}_k$'s are the same on both sides. 

On the right-hand side, the coefficient of $e_{k}^{(i)}$ is 
$$
(-1)^{k} \left(p(X_i)^{n-k}-p(X_i)^{n-k-1} q(X_i) \right),
$$
while it is 
$$
\sum_{j=1}^n (a_{j,k}+a_{j,k+1}q(X_i)) p(X_i)^{\frac{n}{2}}(q(X_i)^{-j}-p(X_i)^{-j}) 
$$
on the left-hand side. For the case $k \ge \frac{n}{2}$, it is easy to see that the coefficients 
coincide and thus we assume 
$k < \frac{n}{2}$ from now on.

We need to show 
$$
\sum_{j=1}^n (a_{j,k}+ a_{j,k+1} q(X)) (q(X)^{-j}-p(X)^{-j}) =
(-1)^k (p(X)^{\frac{n}{2}-k}-p(X)^{\frac{n}{2}-k-1} q(X))
$$
for $k=0,1,\ldots,n-1$. 
For this, it suffices to show 
$$
\sum_{j=1}^n a_{j,k} \left(q(X)^{-j}-p(X)^{-j} \right) = (-1)^k \left( p(X)^{\frac{n}{2}-k} - q(X)^{\frac{n}{2}-k} \right),
$$
which is equivalent to 
$$
\sum_{j \ge 1} \left( \binom{d+j}{2d-j} + \binom{d-1+j}{2d-1-j}\right) \left(q(X)^{-j}-p(X)^{-j} \right) = p(X)^{d} - q(X)^{d},
$$
where we have set $d=\frac{n}{2}-k$ and which we need to show for $d \ge 1$. 

In order to show this, it suffices to show the following two identities 
\begin{align*} 
\sum_{j=1}^{2d} \left( \binom{d+j}{2d-j} + \binom{d+j-1}{2d-j-1} \right) q(X)^{-j} &= (-1)^d q(X)^d X^{3d} + q(X)^{-2d} X^{-3d} \\
\sum_{j=1}^{2d} \left( \binom{d+j}{2d-j} + \binom{d+j-1}{2d-j-1} \right) p(X)^{-j} &= (-1)^d p(X)^d X^{-3d}+ p(X)^{-2d} X^{3d},
\end{align*}
where we have used $q(X)=-p(X) X^{-3}$.
Now, as $q(X)=p(X^{-1})$, it is obvious that the two identities are equivalent and so it suffices to show the second.

We multiply the identity with $p(X)^{2d}$ and then set $m=2d-j$. 

Setting $D=3d$ and 
$$
s(D)= \sum_{m=0}^{\frac{D}{2}}  \binom{D-m}{m} p(X)^{m}, 
$$
we observe that it suffices to show 
 $$
s(D) + s(D-2) p(X) = (1-X)^{D}+X^D
$$
for any positive integer $D$, after checking that the change of the summations bounds is 
eligible. 

The above recursion follows from the next recursion, which is implied by the identity 
$\binom{n}{k}=\binom{n-1}{k}+\binom{n-1}{k-1}$ and valid for any positive integer $D$, by induction with respect to $D$.
$$
s(D) = s(D-1) + s(D-2) p(X)
$$
 \end{proof} 
 
\begin{proof}[Proof of Lemma~\ref{basecase}]
We use induction with respect to $n$. The base case is easy to see. 
 
As for the induction step, we first consider the case $n$ is odd. 
By the definition of $D_s(X_1,\ldots,X_n)$, we have 
\begin{multline*} 
D_{ \frac{n+1}{2} }(X_1,\ldots,X_n) = 
\asym_{X_1,\ldots,X_n} \left[ 
p(X_1)^{ \frac{n-1}{2}}  q(X_1)^{-n} \prod_{j=1}^n \left( p(X_j)-q(X_1) \right) \right. \\
\left. \times 
\prod_{i=2}^{n} p(X_i)^{ \frac{n-1}{2} -(i-1)} q(X_i)^{2(i-1)-2  \left( \frac{n-1}{2}\right) -1}
\prod_{2 \le i \le j \le n} (p(X_j)-q(X_i)) \right] \\
= \sum_{i=1}^{n} (-1)^{n+i} p(X_i)^{ \frac{n-1}{2} } q(X_i)^{-n} \prod_{k=1}^n \left( p(X_k)-q(X_i) \right)
D_{\frac{n-1}{2}}(X_1,\ldots,\widehat{X_i},\ldots,X_n).
\end{multline*} 
By induction and by Laplace expansion, this is equal to 
\begin{multline} 
\label{nodd} 
\prod_{i=1}^n p(X_i)^{-1} \det_{1 \le i, j \le n}
\left( \begin{cases} p(X_i)^{ \frac{n+1}{2}} \left( q(X_i)^{-j} - p(X_i)^{-j} \right) & j \le n-1 \\
        p(X_i)^{ \frac{n+1}{2} } q(X_i)^{-n} \prod_{k=1}^n \left( p(X_k)-q(X_i) \right) & j=n \end{cases} \right).
\end{multline} 
We have completed the proof of this case as soon as we have shown that the difference of the last column of the matrix underlying the determinant in the 
lemma multiplied with $\prod_{i=1}^n p(X_i)$ and the last column of the above matrix can be written as a linear combination of the first $n-1$ columns.
Indeed, starting with the last column of the above matrix, we have 
$$
p(X_i)^{ \frac{n+1}{2} } q(X_i)^{-n} \prod_{k=1}^n \left( p(X_k)-q(X_i) \right) \\ 
=  p(X_i)^{ \frac{n+1}{2}} \sum_{j=0}^n e_j(p(X_1),\ldots,p(X_n)) (-1)^{n+j} q(X_i)^{-j}.
$$
As 
$$
p(X_i)^{ \frac{n+1}{2} }  \sum_{j=0}^n e_j(p(X_1),\ldots,p(X_n)) (-1)^{n+j} p(X_i)^{-j} = 
p(X_i)^{ \frac{n+1}{2}-n} \prod_{j=1}^n \left( p(X_j)-p(X_i) \right)=0,
$$
the above is equal to 
\begin{multline*} 
p(X_i)^{ \frac{n+1}{2} }  \sum_{j=0}^n e_j(p(X_1),\ldots,p(X_n)) (-1)^{n+j} \left( q(X_i)^{-j}- p(X_j)^{-j} \right) \\ =
p(X_i)^{ \frac{n+1}{2} }  \sum_{j=1}^{n-1} e_j(p(X_1),\ldots,p(X_n)) (-1)^{n+j} \left( q(X_i)^{-j}- p(X_j)^{-j} \right) \\
+ \left[ \prod_{j=1}^n p(X_j) \right] p(X_i)^{ \frac{n+1}{2}} \left( q(X_i)^{-n}- p(X_j)^{-n} \right)
\end{multline*}
and we are done with the case $n$ is odd, as the sum over $j \in \{1,\ldots,n-1\}$ in the last expression provides a linear combination of the first 
$n-1$ columns of the matrix underlying the determinant in \eqref{nodd}, and thus the last column of the matrix there can be replaced by the last term 
in this expression, which is just the last column of the matrix underlying the determinant in the claim multiplied by $\prod_{j=1}^n p(X_j)$, 
which cancels the prefactor in \eqref{nodd}. 

Next we consider the case $n$ is even. In this case, 
\begin{multline*} 
D_{\lfloor \frac{n+1}{2} \rfloor}(X_1,\ldots,X_n) =  \asym_{X_1,\ldots,X_n} 
\left[ \prod_{i=1}^{\frac{n}{2}} p(X_i)^{\frac{n}{2}-i} q(X_i)^{2i-n-1} \right.
\\ \left. 
\times \prod_{1 \le i \le j \le n-1} (p(X_j)-q(X_i)) \prod_{k=1}^n (p(X_n)-q(X_k)) \right] \\
= \sum_{i=1}^n  (-1)^{n+i} \prod_{k=1}^n \left(p(X_i)-q(X_k) \right) D_{\frac{n}{2}}(X_1,\ldots,\widehat{X_i},\ldots,X_n) \\
=  \det_{1 \le i, j \le n}
\left( \begin{cases} p(X_i)^{\frac{n}{2}} \left( q(X_i)^{-j} - p(X_i)^{-j} \right) & j \le n-1 \\
         \prod_{k=1}^n \left( p(X_i)-q(X_k) \right) & j=n \end{cases} \right).
\end{multline*} 
Now we use Lemma~\ref{coeffbase} to again write the last column of the matrix underlying this determinant as a linear combination of 
the first $n-1$ columns and the last column of the matrix underlying the determinant, where the coefficient of the last column needs to be 
$1$. The latter is the case as $a_{n,k}=[k=0]$. 
\end{proof}

Before we can complete the proof of Theorem~\ref{antisymtodet}, we still need the following lemma.

\begin{lemma}
\label{generalcoeff}
For any $i,n,s$ with $1 \le i \le n$ and  
$0 \le s \le \lfloor \frac{n+1}{2} \rfloor$, we have 
\begin{multline*} 
\sum_{j=1}^n \sum_{k=0}^n b_{j,k} e_k(q(X_1),\ldots,q(X_n)) p(X_i)^s \left( p(X_i)^{j-(1+[j<2s])s}-q(X_i)^{j-(1+[j<2s])s} \right) \\
=\prod_{k=1}^n (p(X_i) - q(X_k)),
\end{multline*}                   
where 
$$
b_{j,k} = 
\begin{cases} 
(-1)^{k+1} \left( \binom{s+n-k-j}{2 (-2s+n-k)+j} 
+ \binom{s+n-k-j-1}{2 (-2s+n-k)+j-1} \right),
& n-2s+1 \le k \le n -s-1 \,  \&   \, j \le 2s-1, \\
(-1)^{k} [j+k=n+s], &  n - s-1 < k \, \&   \, j \le 2s-1, \\
(-1)^k [j+k=n], & \text{otherwise}. 
\end{cases} 
$$

\end{lemma} 

\begin{proof}
As in the proof of Lemma~\ref{coeffbase}, we expand the product on the right-hand side in the statement of the lemma, then apply 
the identity $e_k=e_k^{(i)}+q(X_i) e_{k-1}^{(i)}$ (using the same notation) and compare the coefficients of 
the $e^{(i)}_k$'s on both sides. Thus we need to show 
\begin{multline*} 
(-1)^{k} \left(p(X)^{n-k}-p(X_i)^{n-k-1} q(X) \right) \\
=
\sum_{j=1}^n (b_{j,k}+b_{j,k+1}q(X)) p(X)^{s} \left( p(X)^{j-(1+[j<2s])s}-q(X)^{j-(1+[j<2s])s} \right)
\end{multline*} 
for $k=0,1,\ldots,n-1$. It suffices to show 
$$ 
(-1)^{k} \left(p(X)^{n-k-s}-q(X)^{n-k-s} \right) \\
=
\sum_{j=1}^n b_{j,k} \left( p(X)^{j-(1+[j<2s])s}-q(X)^{j-(1+[j<2s])s} \right).
$$
The sum on the right-hand side is split into the range $1 \le j \le 2s-1$ and $2s \le j \le n$. 
As for the second part, note that 
$$
\sum_{j=2s}^n (-1)^k [j+k=n] \left( p(X)^{j-s}-q(X)^{j-s} \right)
= [n-k \ge 2s] (-1)^k \left( p(X)^{n-k-s}-q(X)^{n-k-s} \right),
$$
so that we need to show 
$$ 
[n-k \le 2s-1] (-1)^{k} \left(p(X)^{n-k-s}-q(X)^{n-k-s} \right) \\
=
\sum_{j=1}^{2s-1} b_{j,k} \left( p(X)^{j-2s}-q(X)^{j-2s} \right).
$$
Next we distinguish between the cases $k \le n-2s$, $n-2s+1 \le k \le n-s-1$ and $n-s \le k$.
If $k \le n-2s$, then the left-hand side vanishes. We have $b_{j,k}=(-1)^k [j+k=n]$ in this case, so that 
$b_{j,k} \not= 0$ only if $j=n-k$, but since $n-k \ge 2s$, this $j$ is not in the range. If 
$n-s \le k$, then the left hand side is 
$$
(-1)^{k} \left(p(X)^{n-k-s}-q(X)^{n-k-s} \right) 
$$ 
unless $s=0$, but this case is easy to check. Then we have $b_{j,k}=(-1)^k [j+k=n+s]$, so that the 
right-hand side is equal to 
$$
(-1)^k \left( p(X)^{n+s-k-2s}-q(X)^{n+s-k-2s} \right)
$$
which matches the left-hand side, 
since $1 \le n+s-k \le 2s-1$ unless $n+s-k=2s$. If $n+s-k=2s$ then the right-hand side is zero, but so is the left-hand side.

It remains to show 
\begin{multline*} 
\left(p(X)^{n-k-s}-q(X)^{n-k-s} \right) \\ = \sum_{j=1}^{2s-1}  \left( \binom{s+n-k-j}{2 (-2s+n-k)+j} 
+ \binom{s+n-k-j-1}{2 (-2s+n-k)+j-1} \right) \left( q(X)^{j-2s}-p(X)^{j-2s} \right)
\end{multline*} 
if $n-2s+1 \le k \le n-s-1$, which we will show for any integer $s$. For this, it suffices to show the following two identities.
\begin{multline*} 
\sum_{j=1}^{2s-1} \left( \binom{s+n-k-j}{2(-2s+n-k)+j} + \binom{s+n-k-j-1}{2(-2s+n-k)+j-1} \right) p(X)^{j-2s} \\ = 
(-1)^{n+s+k} p(X)^{n-k-s} X^{-3(n-k-s)} + p(X)^{-2n+2k+2s} X^{-3(k-n+s)} 
\end{multline*} 
\begin{multline*} 
\sum_{j=1}^{2s-1} \left( \binom{s+n-k-j}{2(-2s+n-k)+j} + \binom{s+n-k-j-1}{2(-2s+n-k)+j-1} \right) q(X)^{j-2s} \\= 
(-1)^{n+s+k} q(X)^{n-k-s} X^{3(n-k-s)} + q(X)^{-2n+2k+2s} X^{3(k-n+s)}
\end{multline*} 
To see this, we have used $q(X)=-p(X) X^{-3}$. As $q(X)=p(X^{-1})$, the two identities are equivalent and it suffices to 
show the first.

We perform the index transformation $i=2s-j$ in the first identity and obtain the following equivalent identity.
\begin{multline*} 
\sum_{i=1}^{2s-1} \left( \binom{-s+n-k+i}{2(-s+n-k)-i} + \binom{-s+n-k+i-1}{2(-s+n-k)-i-1} \right) p(X)^{-i} \\ = 
(-1)^{n+s+k} p(X)^{n-k-s} X^{-3(n-k-s)} + p(X)^{-2n+2k+2s} X^{-3(k-n+s)}
\end{multline*} 
The sum on the left-hand side does not change if we extend the summation over all $i \ge 1$. Also we set $d=-s+n-k$ and obtain 
the following equivalent identity 
$$
\sum_{i \ge 1} \left( \binom{d+i}{2d-i} + \binom{d+i-1}{2d-i-1} \right) p(X)^{-i} \\ = 
(-1)^{d} p(X)^{d} X^{-3 d} + p(X)^{-2d} X^{3d}
$$
This identity was proved in Lemma~\ref{coeffbase}. 
\end{proof} 

\begin{proof}[Proof of Theorem~\ref{antisymtodet}]
The proof is by induction with respect to $n$ and the base case of the induction, which is $n=2s-1$, is taken care of in Lemma~\ref{basecase}. Assume 
$n>2s-1$. Then, 
\begin{multline*}
D_s(X_1,\ldots,X_n) =  \asym_{X_1,\ldots,X_n} \left[ \prod_{i=1}^s p(X_i)^{s-i} q(X_i)^{2i-2s-1} \right. \\
\left. \times 
\prod_{1 \le i \le j \le n-1} (p(X_j)-q(X_i)) \prod_{k=1}^n (p(X_n)-q(X_k)) \right] \\
= \sum_{i=1}^{n} (-1)^{n+i} \prod_{k=1}^n (p(X_i)-q(X_k)) D_s(X_1,\ldots,\widehat{X_i},\ldots,X_n). 
\end{multline*} 
By induction, this is equal to 
$$
\det_{1 \le i, j \le n}
\left( \begin{cases} p(X_i)^s \left( p(X_i)^{j-(1+[j<2s])s}-q(X_i)^{j-(1+[j<2s])s} \right)  & j \le n-1 \\
         \prod_{k=1}^n \left( p(X_i)-q(X_k) \right) & j=n \end{cases} \right).
$$
The result then follows from Lemma~\ref{generalcoeff} in the same way as in the proof of Lemma~\ref{basecase}. \end{proof}

\section{A generalization of the Cauchy determinant}
\label{cauchysec} 

Recall that the Cauchy determinant evaluates to the following simple product.
\begin{equation}
\label{cauchy}  
\det_{1 \le i,j \le n} \left( \frac{1}{1-X_i Y_j} \right) = \frac{\prod_{1 \le i < j \le n} (X_j-X_i)(Y_j-Y_i)}{\prod_{i,j=1}^n (1-X_i Y_j)}
\end{equation} 
We present a common generalization of this evaluation and Lemma~\ref{first}.

\begin{theorem} 
\label{cauchyg} Let $n$ be a positive integer and $f(X)$, say, a rational function in $X$, then
\begin{multline*} 
\asym_{X_1,\ldots,X_n} \left[ \prod\limits_{1 \le i \le j \le n} \frac{f(X_j)-X_i}{(1-Y_i f(X_j))(1-X_i Y_j)} \right] = 
\frac{\det_{1 \le i, j \le n} \left(\frac{f(X_i)-X_i}{(1-X_i Y_j)(1-f(X_i)Y_j)}\right)}{\prod_{1 \le i < j \le n}(Y_j-Y_i)} \\
= \det_{1 \le i, j \le n} \left( \frac{f(X_j)^i}{\prod_{k=1}^n (1-f(X_j)Y_k)} - \frac{X_j^i}{\prod_{k=1}^n (1-X_jY_k)} \right)
\end{multline*} 
\end{theorem}  

Note that \eqref{cauchy} follows from the second equality in the statement of Theorem~\ref{cauchyg} by setting $f(X)=0$ and using the Vandermonde determinant evaluation to simplify the third expression. Lemma~\ref{first} follows on the other hand by considering the first and the third expression and setting $Y_i=0$.
Also note that the theorem generalizes \cite[Theorem~17]{gog}.
 
\begin{proof}
We start by showing the equality of the second and third expression.
By transposing and partial fraction decomposition, the second expression can be written as 
$$
\frac{\det_{1 \le i, j \le n} \left(\frac{f(X_j)}{1-f(X_j) Y_i} - \frac{X_j}{1-X_j Y_i}\right)}{\prod_{1 \le i < j \le n}(Y_j-Y_i)}  = 
\frac{\det_{1 \le i, j \le n} \left(\sum_{k \ge 0} (f(X_j)^{k+1}-X_j^{k+1}) Y_i^k \right)}{\prod_{1 \le i < j \le n}(Y_j-Y_i)}. 
$$
By \eqref{limit1g}, this is equal to 
\begin{equation}
\label{jacobitrudi} 
\det_{1 \le i, j \le n} \left(\sum_{k \ge 0} (f(X_j)^{i+k}-X_j^{i+k}) h_k(Y_1,\ldots,Y_n) \right) = 
\det_{1 \le i, j \le n} \left( \frac{f(X_j)^i}{\prod_{k=1}^n (1-f(X_j)Y_k)} - \frac{X_j^i}{\prod_{k=1}^n (1-X_jY_k)} \right).
\end{equation}

Next we show the equality of the first and the second expression. We multiply both expressions 
by $\prod_{1 \le i < j \le n}(Y_j-Y_i)$. 
We derive a recursion for the first expression in the same way as in the proof of Lemma~\ref{basecase} and Theorem~\ref{antisymtodet}, and then show that \eqref{jacobitrudi} satisfies the same recursion.

Concretely, let
$$A(X_1,\ldots,X_n;Y_1,\ldots,Y_n)=\asym_{X_1,\ldots,X_n} \left[ \prod\limits_{1 \le i \le j \le n} \frac{f(X_j)-X_i}{(1-Y_i f(X_j))(1-X_i Y_j)} \right]\prod_{1\leq i<j\leq n}(Y_j-Y_i),$$
then 
\begin{multline*} 
A(X_1,\ldots,X_n;Y_1,\ldots,Y_n) \\= \sum_{j=1}^{n} (-1)^{j+1} 
\prod_{i=2}^n (Y_i-Y_1) \prod_{i=1}^n \frac{f(X_i)-X_j}{(1-Y_1 f(X_i))(1-X_j Y_i)} 
A(X_1,\ldots,\widehat{X_j},\ldots,X_n;Y_2,\ldots,Y_n).
\end{multline*} 
We need to show that 
\begin{equation} 
\label{B}
\det_{1 \le i, j \le n} \left(\frac{f(X_i)-X_i}{(1-X_i Y_j)(1-f(X_i)Y_j)} \right)
\end{equation} 
satisfies the same recursion. This amounts to showing that 
\begin{multline*} 
 \det_{1 \le i, j \le n} \left(\frac{f(X_i)-X_i}{(1-X_i Y_j)(1-f(X_i)Y_j)} \right) \\
= \sum_{q=1}^{n} (-1)^{q+1} 
\prod_{p=2}^n (Y_p-Y_1) \prod_{p=1}^n \frac{f(X_p)-X_q}{(1-Y_1 f(X_p))(1-X_q Y_p)}  
\det_{1 \le i \le n, i \not=q \atop 2 \le j \le n} \left(\frac{f(X_i)-X_i}{(1-X_i Y_j)(1-f(X_i)Y_j)} \right). 
\end{multline*} 
By Laplace expansion, the right-hand side is equal to 
\begin{equation} 
\label{laplace} 
\det_{1 \le i,j \le n} \left(
\begin{cases} 
\prod_{p=2}^n (Y_p-Y_1) \prod_{p=1}^n \frac{f(X_p)-X_i}{(1-Y_1 f(X_p))(1-X_i Y_p)} & j=1 \\
\frac{f(X_i)-X_i}{(1-X_i Y_j)(1-f(X_i)Y_j)} & j \ge 2 
\end{cases} 
 \right).
\end{equation}The assertion now follows from 
\begin{multline}
\label{LC} 
\sum_{j=2}^n \frac{\prod_{p \not=1,j}(Y_1-Y_p) \prod_{p=1}^{n} (1-Y_j f(X_p))}
{\prod_{p \not=1,j}(Y_j-Y_p) \prod_{p=1}^{n} (1-Y_1 f(X_p))}
\frac{f(X_i)-X_i}{(1-X_i Y_j)(1-f(X_i)Y_j)} \\ + 
\prod_{p=2}^n (Y_p-Y_1) \prod_{p=1}^n \frac{f(X_p)-X_i}{(1-Y_1 f(X_p))(1-X_i Y_p)} 
= \frac{f(X_i)-X_i}{(1-X_i Y_1)(1-f(X_i)Y_1)}
\end{multline}for all $i \in \{1,\ldots,n\}$. 
To see this, let $W$ denote the matrix underlying the determinant in \eqref{B} and $W^{*}$ denote the matrix underlying the 
determinant in \eqref{laplace}, and note that $W$ and $W^{*}$ coincide except for the first column. In \eqref{LC}, 
we have expressed the first column of $W$ as a linear combination of the other columns of $W$ (or, equivalently, $W^*$) 
as the sum of a linear combination of the other columns and the first column $W^*$. 
As soon as \eqref{LC} is proven, this implies that the determinants of 
$W$ and $W^{*}$ coincide. 
Thus, in order to conclude the proof, we need to show \eqref{LC}. By incorporating the right-hand side as the $j=1$ summand 
on the left-hand side and after several further manipulations, the identity is equivalent to 

\begin{equation*}
	\sum_{j=1}^n\prod_{p\neq i}(Y_jZ_p-1)\prod_{p\neq j}\frac{1-X_iY_p}{Y_j-Y_p}=\prod_{p\neq i}(Z_p-X_i),
\end{equation*} which we need to show for all $i\in\{1,\ldots,n\}$ and where we have set $Z_k=f(X_k)$.
Dividing both sides by $X_i^{n-1}$ yields the equivalent equation 
\begin{equation*}
	\sum_{j=1}^n\prod_{p\neq i}(Y_jZ_p-1)\prod_{p\neq j}\frac{X_i^{-1}-Y_p}{Y_j-Y_p}=\prod_{p\neq i}(X_i^{-1}Z_p-1).
\end{equation*} Here the left-hand side is simply the Lagrange interpolation polynomial of the right-hand side as a polynomial in $X_i^{-1}$ over $\mathbb{Q}(Y_1,\dots,Y_n,Z_1,\dots,Z_n)$ with nodes at $Y_1,\dots,Y_n$. This proves \eqref{LC} and thus finishes the proof.

\end{proof}

There are two obvious questions related to this generalization of Lemma~\ref{first}: First, are there any applications similar to the applications of Lemma~\ref{first}? Second, is there a similar generalization of Theorem~\ref{antisymtodet}, and if so, does this have applications too? 

\section{Acknowledgements} 

We thank Sunil Chhita and Filippo Colomo for useful discussions.

\bibliographystyle{abbrvurl}

\bibliography{ASMLittlewood.bib}

\end{document}